%% file: Orthotropic.extreme.2.tex
\documentclass[12pt]{article}
\usepackage{graphicx}
\usepackage{amsmath}
\usepackage{amsfonts}
\usepackage{amsthm}
\usepackage[T1]{fontenc}
\usepackage{url}
\usepackage{bm}
\usepackage{color}
\usepackage[margin=1in]{geometry}
\input{mymacros}

\begin{document}
\title{A note on the extreme points of the cone of quasiconvex quadratic forms with orthotropic symmetry}
\author{Davit Harutyunyan\thanks{University of California Santa Barbara, harutyunyan@ucsb.edu}}
\maketitle

\begin{abstract}
We study the extreme points of the cone of quasiconvex quadratic forms with linear elastic orthotropic symmetry. We prove that if the determinant of the acoustic matrix of the associated forth order tensor of the quadratic form is an extremal polynomial, then the quadratic form is an extreme point of the cone in the same symmetry class. The extremality of polynomials and quadratic forms here is understood in classical convex analysis sense. 
\end{abstract}


\section{Introduction}
\setcounter{equation}{0}
\label{sec:1}

\textbf{Quasiconvexity and Rank-one convexity.} Quasiconvexity has been a central concept in applied mathematics since the work of Morrey [\ref{bib:Morrey.1},\ref{bib:Morrey.2}]. It is tied with existence of minimizers of integral functionals with the Lagranagian satisfying certain type of growth conditions [\ref{bib:Ball},\ref{bib:Dacorogna}]. One of the equivalent definitions of quasiconvexity reads as follows: \textit{Assume $n,N\in \mathbb N$ and $f\colon \mathbb R^{N\times n}\to\mathbb R.$ If the function $f$ is Borel measurable and locally bounded, then it is said to be quasiconvex, if}

\begin{equation}
\label{1.1}
f(\bm{\xi})\leq \int_{[0,1]^n}f(\bm{\xi}+\nabla\varphi(x))dx,
\end{equation}
\textit{for all matrices $\bm{\xi}\in\mathbb R^{N\times n}$ and all functions $\varphi\in W_0^{1,\infty}([0,1]^n,\mathbb R^N). $}
The condition of rank-one convexity occurs naturally in the second variation of an integral functional 
$$\int_{D}L(\nabla u(x))dx,$$
reducing to a pointwise condition on the Hessian of the Lagrangian $L$; the condition is known as the Legendre-Hadamard condition [\ref{bib:VanHove.1},\ref{bib:VanHove.2},\ref{bib:Serre},\ref{bib:Dacorogna}]. The condition of rank-one-convexity reads as follows: \textit{Assume $n,N\in \mathbb N$ and $f\colon \mathbb R^{N\times n}\to\mathbb R.$ If the function $f$ is Borel measurable and locally bounded, then it is said to be rank-one-convex, if}

\begin{equation}
\label{1.2}
f(\lambda\BA+(1-\lambda)\BB)\leq \lambda f(\BA)+(1-\lambda)f(\BB),
\end{equation}
\textit{for all $\lambda\in[0,1]$ and all matrices $\BA,\BB\in\mathbb R^{N\times n}$ such that $\mathrm{rank}(\BA-\BB)\leq 1.$ }
For $C^2$ functions $f$ the condition (\ref{1.2}) is equivalent to the Legandre-Hadamard condition mentioned above and is given by 

\begin{equation}
\label{1.3}
\sum_{\substack{0\leq\alpha,\gamma\leq N \\ 0\leq \beta,\delta\leq n}}\frac{\partial^2 f(\bm{\xi})}{\partial\xi_{\alpha\beta}\partial\xi_{\gamma\delta}}x_\alpha y_\beta x_\gamma y_\delta\geq 0,
\end{equation}
for all $\Bx=(x_1,x_2,\dots,x_N)\in\mathbb R^N$ and $\By=(x_1,x_2,\dots,x_n)\in\mathbb R^n.$
A special choice of the test function $\varphi$ in (\ref{1.1}) proves that in fact quasiconvexity implies rank-one convexity [\ref{bib:Dacorogna}]. 
It is know that when the function $f\colon\mathbb R^{N\times n}\to\mathbb R$ is a quadratic form, then the quasiconvexity and rank-one-convexity of $f$ become equivalent conditions [\ref{bib:Dacorogna}]. For a quadratic form $f(\bm{\xi})$ defined on $N\times n$ matrices, i.e., $f\colon\mathbb R^{N\times n}\to\mathbb R,$ the quasiconvexity of $f$ is equivalent to the condition 
$f(\bm{\xi})\geq 0$ for all $\bm{\xi}\in M^{N\times n}$ such that $\mathrm{rank}(\bm{\xi})\leq 1,$ which amounts to the inequality 
\begin{equation}
\label{1.4}
f(\Bx\otimes\By)\geq 0,\quad\text{for all vectors} \quad \Bx\in\mathbb R^N, \By\in\mathbb R^n.
\end{equation}
\textbf{Positive semidefinite biquadratic forms, quasiconvex quadratic forms: historical facts.} It is clear that the expression for $f(\Bx\otimes\By)$ in (\ref{1.4}) is a biquadratic form in the variables $x_i,$ $i=1,2,\dots,N$ and $y_j$, $j=1,2,\dots,n.$ Biquadratic forms $f$ satisfying the condition (\ref{1.4}) are called positive semidefinite in linear algebra. Thus any quasiconvex quadratic form can be regarded as a positive semidefinite biquadratic form and vice versa. Interestingly enough, the study of nonnegative biquadratic forms has been of 
interest of two mathematical communities: 1. Applied Mathematics/Calculus of Variations, where as mentioned above one adopts the language of quasiconvex forms coming from minimization problems in the calculus of variations and applications in applied mathematics. 2. Linear Algerba/Algebraic Geometry, where one uses the language of positive semidefinite biquadratic forms, and  the interest seems to be purely mathematical with some deep consequences. For the sake of brevity, we will conventionally number those communities first and second. While quasiconvex quadratic forms gained tremendous interest after the work of Morrey [\ref{bib:Morrey.1}] in 1952, the study of positive semidefinite biquadratic forms has started as early as the work of Hilbert [\ref{bib:Hilbert}]. In 1888 Hilbert raised the question of whether any nonnegative polynomial over reals can be expressed as a sum of squares of rational functions, which was solved in affirmative by Artin in 1927, [\ref{bib:Artin}]. Hilbert also showed that any nonnegative homogeneous polynomial of fourth degree in three variables can be expressed as a sum of squares of second order homogeneous polynomials [\ref{bib:Hilbert}]. Denote the convex cone of nonnegative biquadratic forms $f(\Bx\otimes\By)$ by $\cal {C}$$_{N,n},$ where $\Bx\in\mathbb R^N,$ and 
$\By\in\mathbb R^n$. While the second community has been studying general classical questions on nonnegative biquadratic forms, such as extreme points of $\cal {C}$$_{N,n}$
 [\ref{bib:Choi},\ref{bib:Cho.Lam.}], separability and inseparability of positive linear maps [\ref{bib:Stormer}], etc., the first community has been mostly interested in certain questions coming from applications in materials science, among which is the study of quasiconvex and polyconvex quadratic forms and whether they are different in general. Recall that given $n,N\in\mathbb N,$ a function $f\colon\mathbb R^{N\times n}\to\mathbb R$ is called polyconvex if there exists a convex function $g\colon\mathbb R^k\to\mathbb R,$ such that $f(\bm{\xi})=g(m_1(\bm{\xi}),m_2(\bm{\xi}),\dots,m_k(\bm{\xi})),$ where $k$ is the number of all possible minors of the matrix $\bm{\xi}\in\mathbb M^{N\times n},$ and $m_i(\bm{\xi})$ are all minors of $\bm{\xi},$ for $i=1,2,\dots,k.$ Unfortunately there has been little communication between the two communities mentioned above; below are some historical facts that are worth mentioning. As mentioned above, a quadratic form $f\colon R^{N\times n}\to\mathbb R$ is quasiconvex iff (\ref{1.4}) is satisfied. Also, it has long been known that $f(\bm{\xi})$ is polyconvex if it can be written as a sum a convex form and a linear combination of the $2\times 2$ minors of $\bm{\xi}.$ Van Hove [\ref{bib:VanHove.1},\ref{bib:VanHove.1}] proved in the late 1940's that for $n,N\geq 3,$ there exist quasiconvex quadratic forms $f\colon R^{N\times n}\to\mathbb R$ that are not polyconvex. It has been believed in the first community that the first explicit example of such a quadratic form was provided by Serre [\ref{bib:Serre}] in 1980, given by 
\begin{equation}
\label{1.5}
f(\bm{\xi})=(\xi_{11}-\xi_{23}-\xi_{32})^2+(\xi_{12}-\xi_{31}+\xi_{13})^2+(\xi_{21}-\xi_{13}-\xi_{31})^2+\xi_{22}^2+\xi_{33}^2-\epsilon |\bm{\xi}|^2,
\end{equation}
which is quasiconvex but not polyconvex for small enough $\epsilon>0.$ However, as a referee\footnote{Evidently from the second community} of the current paper pointed out, the first such explicit example is due to Choi [\ref{bib:Choi}] delivered in 1975:
\begin{equation}
\label{1.6}
f(\bm{\xi})=\xi_{11}^2+\xi_{22}^2+\xi_{33}^2-2(\xi_{11}\xi_{22}+\xi_{22}\xi_{33}+\xi_{33}\xi_{11})+2(\xi_{12}^2+\xi_{23}^2+\xi_{31}^2),
\end{equation}
which is very well known in the second community. Another interesting fact is that shortly after Choi's paper [\ref{bib:Choi}] appeared, Choi and Lam [\ref{bib:Cho.Lam.}] provided 
an example of a quasiconvex quadratic form (for $n=N=3$), that is en extreme point of the convex cone $\cal {C}$$_{3,3}$ in classical convex analysis sense:
\begin{equation}
\label{1.7}
f(\bm{\xi})=\xi_{11}^2+\xi_{22}^2+\xi_{33}^2-2(\xi_{11}\xi_{22}+\xi_{22}\xi_{33}+\xi_{33}\xi_{11})+\xi_{12}^2+\xi_{23}^2+\xi_{31}^2.
\end{equation}
While the Choi-Lam form in (\ref{1.7}) is widely known within the second community as the first nontrivial\footnote{Certainly the square of a linear form is a trivial example of such} explicit example of an extreme point of $\cal {C}$$_{3,3}$, is was not known in the first community until a referee of the current article pointed it out, and it was rediscovered by the author and Milton in 
[\ref{bib:Har.Mil.1}].\\ 
\textbf{Extremal quasiconvex quadratic forms, extremals in the sense of Milton.} In the present work we continue the study of the so called \textit{extremal quasiconvex quadratic forms} introduced by Milton in [\ref{bib:Milton.1}]. Those are quasiconvex quadratic forms that lose the quasiconvexity property whenever a convex quadratic form is subtracted from them. Let us highlight, that 
these and the extreme points (rays) of $\cal {C}$$_{N,n}$ are different in general, as for instance the square of a linear form in $\bm{\xi}$ is an extremal ray of $\cal {C}$$_{N,n}$ 
in classical convex analysis sense, whereas it is not an extremal in the sense of Milton. Another trivial fact is that if the biquadratic form $f(\Bx\otimes\By)$ is not a 
perfect square and is an extreme ray of the cone, then the associated quadratic form $f(\bm{\xi})$ will be an extremal in the sense of Milton. A related interesting question that seems to be open is the following: \textit{is there a quasiconvex quadratic form that is an extremal in the sense of Milton, but fails to be an extremal in classical convex analysis sense?}
 Milton showed in [\ref{bib:Milton.3}], that such extremal forms (in the sense of Milton) may be used to derive new bounds on the effective properties of composites using the translation method, as done for the gradient problem in  [\ref{bib:Mur.Tar.},\ref{bib:Tartar},\ref{bib:All.Koh.},\ref{bib:Che.Gib.1},\ref{bib:Che.Gib.2},\ref{bib:Koh.Lip.},\ref{bib:Milton.2},\ref{bib:Cherkaev},\ref{bib:Mil.Ngu.},\ref{bib:Kan.Kim.Mil.},\ref{bib:Kan.Mil.Wan.}] using Null-Lagrangians. Here, Null-Lagrangians are the linear combinations of second order minors of the argument matrix $\bm{\xi}$ of the function $f\colon\mathbb R^{N\times n}\to\mathbb R.$ In the work [\ref{bib:Kan.Mil.}], Kang and Milton use special type of extremals to obtain bounds on the volume fraction of two materials in three dimensions. In [\ref{bib:Milton.3}], Milton introduced a generalization of the notion of quasiconvexity, where extremals start to play a key role. The idea is to consider the quasiconvexity inequality (\ref{1.1}) with $\cal{F}(\varphi)$ instead of $\nabla\varphi,$ where $\cal{F}$ is in general a differential operator (it is the gradient operator in the case of standard quasiconvexity). Also, the admissible vector fields $\varphi$ may be required to satisfy some partial-differential relation, such as having $\cal{F}(\varphi)=\nabla\varphi$ with the additional condition that $\mathrm{div}(\varphi)=0$ for $N=n.$ Then, as suggested by Milton [\ref{bib:Milton.3}], the extremals may replace the Null-Lagrangians for general operators $\cal{F}$ and differential constraints imposed on $\varphi.$
Of course when considering the quasicovexity property of quadratic forms, a quadratic form $Q(\bm{\xi})$ ($\bm{\xi}\in\mathbb R^{N\times n}$) is considered modulo Null-Lagrangians, as in the rank-one-convexity condition (\ref{1.4}) Null-Lagrangians simply vanish and do not play a role. Considering quadratic forms with linear elastic orthotropic symmetry, our work [\ref{bib:Har.Mil.2}] reveals a link between such extremal forms and extremal polynomials, suggesting that the question of extremality of a quadratic form $Q(\bm{\xi})$ is closely related with the extremality of the determinant of the so called acoustic matrix (or the $\By-$matrix) of the biquadratic form $Q(\Bx\otimes\By),$ which is the matrix $T(\By),$ where $Q(\Bx\otimes\By)=\Bx T(\By)\Bx^T.$ This fact (connection) seems to have been unknown also in the second community. Recall that in classical convex analysis, a homogeneous nonnegative polynomial $P$ is called an extremal, if it loses the non-negativity property when another nonnegative homogeneous polynomial of the same degree and other than a scalar multiple of $P$ is subtracted from it. The later work [\ref{bib:Har.Mil.3}] goes deeper and proves that the condition found in [\ref{bib:Har.Mil.2}] applies to any quadratic forms (not necessarily with orthotropic symmetry), see 
[\ref{bib:Har.Mil.3}, Theorems~2.4--2.7]. For the convenience of the reader we recall/formulate Theorem~2.5 of [\ref{bib:Har.Mil.3}].
\begin{theorem}
\label{th:1.1}
Assume $Q(\bm{\xi})\colon\mathbb R^{3\times 3}\to\mathbb R$ is a quasiconvex quadratic form such that $\det(T(\By))$ is an extremal polynomial that is not a perfect square, where $Q(\Bx\otimes\By)=\Bx T(\By)\Bx^T.$ Then $Q(\bm{\xi})$ is an extremal (in the sense of Milton, mentioned above). 
\end{theorem}

Theorems~2.6 and 2.7 in [\ref{bib:Har.Mil.3}] study the case when $\det(T(\By))$ is a perfect square\footnote{It is not difficult to prove that the square of a third order homogeneous polynomial in three variables is en extremal polynomial.}, including the case when it is identically zero, proving that then the form $Q$ must be an an extremal or a polyconvex form, or a sum of those with some additional specific properties. In the present work, in contrast to the works [\ref{bib:Har.Mil.2},\ref{bib:Har.Mil.3}], we will not impose any additional constraint on $\det(T(\By))$ other than extremality. We will solely work with quadratic forms with orthotropic symmetry due to the interest in elasticity, however our analysis applies to other classes of quasiconvex quadratic forms too as it will become clear later in Section~\ref{sec:4}. In the meantime we will study a deeper problem here, namely not the extremals in the sense of Milton, that lose the quasiconvexity property when a convex form is subtracted from them, but the extreme points of the convex cone of the appropriate class of quasiconvex quadratic forms in the classical convex analysis sense. We will prove that the same link found in [\ref{bib:Har.Mil.2}] between quadratic forms and their $\By-$matrix determinant is still present, see Theorem~\ref{th:3.1}.

Another related problem concerning sixth order homogeneous polynomials in three variables and determinants of $\By-$matrices of quadratic forms, namely whether or not any such polynomial, in particular the well known Robinson's polynomial is a determinant of a $\By-$matrix, is open as well [\ref{bib:Reznick},\ref{bib:Quarez}]. In [\ref{bib:Buc.Siv.}] the authors construct the first examples of nonnegative biquadratic forms with a tensor in $\mathbb (R^3)^4,$ that have maximal number of nontrivial zeroes, namely ten of them. Note that the later can be regarded as quasocinvex quadratic forms that are extreme points of the appropriate cone as their $\By-$matrix determinants are scalar multiples of the generalized Robinson's polynomial [\ref{bib:Buc.Siv.},\ref{bib:Reznick}]. In conclusion we mention that while in the papers [\ref{bib:Reznick},\ref{bib:Quarez},\ref{bib:Buc.Siv.},\ref{bib:Cho.Kye.Lee.},\ref{bib:Li.Wu.},\ref{bib:Hou.Li.Poo.Qi.Sze.}], and in the references therein the approach to the problem of extremals and extreme points is algebraic and algebraic geometric (in some of them the authors use the language of positive linear maps), we regard it in the context of applied mathematics and adopt a somewhat simple and elementary approach to it.\\

\section{Quadratic forms with orthotropic symmetry}
\setcounter{equation}{0}
\label{sec:2}

In this section we recall the definition of orthotropic materials, i.e., quadratic forms that have orthotropic symmetry in linear elasticity. A homogeneous orthotropic elastic material has three mutually orthogonal planes such that the material properties are symmetric under reflection about each plane. If cartesian coordinate axes are chosen orthogonal to these planes, then the properties are invariant under the transformations $x_a \to -x_a$, $x_b \to x_b$, and $x_c\to x_c$, where $abc$ is a permutation of $123$. Recall that in linear elasticity the stress tensor $\BGs$ is related to the strain vector $\BGe$ linearly via a forth order constant tensor (the elasticity tensor) $\BC:$

\begin{equation}
\label{2.1}
\BGs=\BC\cdot\BGe.
\end{equation}
Due to the symmetry of the strain tensor, the relation (\ref{2.1}) is then reduced to a similar one with $\BC$ being a six by six matrix and $\BGs$ and $\BGe$ being $6$-vectors (Voigt notation). In the case of orthotropic materials, the elements of the elasticity tensor such as $C_{abcc}$ and $C_{abbb},$ change sign under a reflection about a symmetry plane mentioned above, thus those must be zero. Thus the elements $C_{ijk\ell}$ of the elasticity tensor must be zero unless the indices $ijk\ell$ contain an even number of repetitions of the indices $1$, $2$ or $3$. The Voigt notation takes the form 

\begin{equation}
\label{2.2}
\BGs=\BC\cdot\BGe,\quad\text{where}\quad
\BGs=\begin{bmatrix} \sigma_{11} \\ \sigma_{22} \\ \sigma_{33} \\ \sigma_{23} \\ \sigma_{31} \\ \sigma_{12} \end{bmatrix}, \quad
\BGe=\begin{bmatrix} \epsilon_{11} \\ \epsilon_{22} \\ \epsilon_{33} \\ 2\epsilon_{23} \\ 2\epsilon_{31} \\ 2\epsilon_{12} \end{bmatrix}, \quad
\BC=
\begin{bmatrix}
C_{11} & C_{12} & C_{13} & 0 & 0 & 0\\
C_{12} & C_{22} & C_{23} & 0 & 0 & 0\\
C_{13} & C_{23} & C_{33} & 0 & 0 & 0\\
0 & 0 & 0 & C_{44} & 0 & 0\\
0 & 0 & 0 & 0 & C_{55} & 0\\
0 & 0 & 0 & 0 & 0 & C_{66}
\end{bmatrix}.
\end{equation}
The related quadratic form then has the form 

\begin{align}
\label{2.3}
Q(\bm{\xi})&=\sum_{i,j=1}^3 C_{ij}\epsilon_{ii}\epsilon_{jj}+4C_{44}\epsilon_{23}^2+4C_{55}\epsilon_{31}^2+4C_{66}\epsilon_{12}^2\\ \nonumber
&=\sum_{i,j=1}^3 C_{ij}\xi_{ii}\xi_{jj}+C_{44}(\xi_{23}+\xi_{32})^2+C_{55}(\xi_{31}+\xi_{13})^2+C_{66}(\xi_{12}+\xi_{21})^2,
\end{align}
where the variable $\xi_{ij}$ plays the role of the $ij-$th entry of the displacement gradient $\BGx=\nabla \Bu=(\frac{\partial u_i}{\partial x_j}),$ $i,j=1,2,3.$ Thus 
we have for the entries of the strain matrix $\epsilon_{ij}=\frac{1}{2}(\xi_{ij}+{\xi_{ji}}),$ and (\ref{2.3}) follows. The mechanical properties are in general different along each axis. Orthotropic materials require 9 elastic constants and have as subclasses isotropic materials (with 2 elastic constants), cubic materials (with 3 elastic constants), and transversely isotropic materials (with 5 elastic constants). The wood in a tree trunk is an example of a material which is locally orthotropic: the material properties in three perpendicular
directions, axial, radial, and circumferential, are different. Many crystals and rolled metals are also examples of orthotropic materials.

\section{Main Results}
\setcounter{equation}{0}
\label{sec:3}
the below theorem is the main result of the manuscript. 
\begin{theorem}
\label{th:3.1}
Denote the convex cone of $3\times 3$ quasiconvex quadratic forms with linear elastic orthotropic symmetry by $\cal{C}.$ Assume $Q(\bm{\xi})\in \cal{C},$ where $Q$ having the form shown in (\ref{2.3}):

\begin{equation} 
\label{3.1}
Q(\bm{\xi})=\sum_{i,j=1}^3C_{ij}\xi_{ii}\xi_{jj}+C_{44}(\xi_{23}+\xi_{32})^2+C_{55}(\xi_{13}+\xi_{31})^2+C_{66}(\xi_{12}+\xi_{21})^2,
\end{equation}
satisfies the strict inequalities\footnote{Note that due to the quasiconvexity of $Q,$ the bounds $C_{ii}\geq 0$ for $i=1,2,\dots,6$ must be fulfilled.} 

\begin{equation}
\label{3.2}
C_{ii}>0,\quad\text{ for} \quad i=1,2,\dots,6. 
\end{equation}
If the determinant of the acoustic tensor of $Q$ is an extremal polynomial, then $Q$ is an extreme point of $\cal{C}.$ 
\end{theorem}

\begin{remark}
\label{rem:3.2}
It has been shown in [\ref{bib:Har.Mil.2}, Theorem~5.1], that assuming $C_{11},C_{22},C_{33}>0$ and that the determinant of the acoustic tensor of $Q$ is an extremal polynomial that is not a perfect square, then $Q$ is an extremal in the sense of Milton, i.e., it loses the quasiconvexity property when a rank one form\footnote{A rank-one form is the square of a linear form.} is subtracted from it. We claim that given the additional conditions $C_{44},C_{55},C_{66}>0,$ i.e., the condition (\ref{3.2}), the mentioned determinant automatically can not be a perfect square. Indeed, the form of $\det(T(\By))$ in (\ref{4.16}) suggests that if it is a perfect square, then it has the form $(a_1y_1^3+a_2y_2^3+a_3y_3^3+P(\By))^2,$ where $P(\By)$ is a homogeneous polynomial of degree $3$ in the variable $\By,$ free of any of the monomials $y_1^3,y_2^3,$ and $y_3^3.$ If then a monomial $y_i^2y_j$ with $i\neq j$ occurs in $P(\By),$ then the monomial $y_i^5y_j$ will occur in the determinant, which is not the case due to (\ref{4.16}). Thus $P(\By)$ may only involve $y_1y_2y_3,$ which results in the determinant involving $y_1^4y_2y_3,$ which is again not possible. This being said, due to [\ref{bib:Har.Mil.2}, Theorem~5.1], we automatically obtain that under the assumptions of Theorem~\ref{th:3.1}, the quadratic form $Q(\bm{\xi})$ is also an extremal in the sense of Milton. 
\end{remark}

We also mention that the works [\ref{bib:Cho.Kye.Lee.}] and [\ref{bib:Li.Wu.}] study positivity (quasdiconvexity) and inseparability (extremality) of speacial type of biquadratic forms. In particular
the authors in [\ref{bib:Cho.Kye.Lee.}] charactarize orthotropic quadratic (biquadratic) forms for which the coefficients of the off-diagonal terms $y_iy_j$, $i\neq j$ of $T(\By)$ are $-1$, with the on-diagonal terms having some special form and depending on $3$ parameters only. This is all done of course for the case $N=n=3.$ The work [\ref{bib:Li.Wu.}] treats again some very special type of biquadratic forms in the case $N=n$ and studies them for positive semi-definiteness and in-decomposability (extremality). However, both works [\ref{bib:Cho.Kye.Lee.},\ref{bib:Li.Wu.}] have very little intersection with ours.

\section{Proof of the main result}
\setcounter{equation}{0}
\label{sec:4}

\subsection{A lemma from linear algebra}

Here we prove a lemma from linear algebra, that will be a key factor in the proof of the main results. Although the conclusion of the lemma will be needed only in 3D, we will formulate and prove it for a general $n\in\mathbb N$ as it may be of separate interest. 

\begin{lemma}
\label{lem:4.1}
Let $n\in\mathbb N$ be such that $n\geq 2.$ Assume $\BA,\BB\in \mathbb M_{sym}^{n\times n}$ be symmetric positive semi-definite matrices such that $\BA\geq \BB$ in the sense of quadratic forms. Then for any integers $1\leq k< m\leq n$ one has the inequality 

\begin{equation}
\label{4.1}
\frac{1}{{n \choose m}}\sum_{M_m(\BB)}M_m(\BB)\mathrm{cof}_{\BA}(M_m(\BB))\leq \frac{1}{{n \choose k}}\sum_{M_k(\BB)}M_k(\BB)\mathrm{cof}_{\BA}(M_k(\BB)),
\end{equation}
where the number ${n \choose m}$ is the binomial coefficient, and the sum $\sum_{M_m(\BB)}$ is taken over all $m-$th order minors $M_m(\BB)$ of $\BB,$ and $\mathrm{cof}_{\BA}(M_m(\BB))$ denotes the cofactor of the minor in the matrix $\BA,$ obtained by choosing the same rows and columns as to get the 
minor $M_m(\BB)$ in $\BB.$ 
\end{lemma}

\begin{proof}
We may assume without loss of generality that $\BA$ and $\BB$ are positive definite as we can prove the statement for the matrices $\BA+\epsilon \BI$ 
and $\BB+\epsilon\BI,$ where the parameter $\epsilon$ is positive and then send $\epsilon$ to zero and recover the estimate (\ref{4.1}) for $\BA$ and $\BB.$ Consider next the polynomial 

\begin{equation}
\label{4.2}
P(t)=\det(\BA-t\BB), \qquad t\in \mathbb R.
\end{equation}
Denoting 

$$S_m=\sum_{M_m(\BB)}M_m(\BB)\mathrm{cof}_{\BA}(M_m(\BB)),\qquad m=0,1,2,\dots,n$$
we clearly have that 

\begin{equation}
\label{4.3}
P(t)=\sum_{m=0}^n (-1)^mS_mt^m,\qquad t\in \mathbb R.
\end{equation}
We have on one hand due to the fact $\det(\BB)>0,$ thus the square root $\BB^{1/2}$ and the half-inverse $\BB^{-1/2}$ exist, therefore we have

$$P(t)=\det(\BA-t\BB)=\det (\BB^{1/2}(\BB^{-1/2}\BA\BB^{-1/2}-t\BI)\BB^{1/2})=\det(\BB)\det(\BB^{-1/2}\BA\BB^{-1/2}-t\BI),$$
thus the roots of $P$ are real as $P$ is a scalar multiple of the characteristic polynomial of the symmetric matrix 
$\BB^{-1/2}\BA\BB^{-1/2}.$ On the other hand the inequalities $\BA\geq \BB$ and $\det(\BB)>0$ imply that $P(t)>0$ for $t\in (-\infty,1).$ Consequently we obtain that the roots of $P$ are real and belong to the interval $[1,\infty).$ Denoting them by $t_1,t_2,\dots, t_n$ we have 

$$P(t)=(-1)^n\det(\BB)(t-t_1)(t-t_2)\cdots(t-t_n),$$
which together with (\ref{4.3}) and Vieta's theorem gives the formulae

\begin{equation}
\label{4.4}
S_m=\det(\BB)\sum_{1\leq i_1<i_2<\dots<i_{n-m}\leq n}{t_{i_1}}t_{i_{2}}\cdots t_{i_{n-m}},\quad m=0,1,\dots,n,
\end{equation}
where by convention $S_0=\det(\BB).$ The formula (\ref{4.4}) now reduces (\ref{4.1}) to the inequality 
\begin{equation}
\label{4.5}
\frac{1}{{n \choose q}} \sum_{1\leq i_1<i_2<\dots<i_q\leq n} t_{i_1}t_{i_2}\cdots t_{i_q}\leq 
\frac{1}{{n \choose p}} \sum_{1\leq i_1<i_2<\dots<i_p\leq n} t_{i_1}t_{i_2}\cdots t_{i_p}, \quad \text{for all}\quad 0\leq q<p\leq n,
\end{equation}
where again in the case $q=0$ the whole expression is understood to be $1$. Assume $n\geq p>q\geq 1,$ the case $q=0$ being obvious by the bounds $t_i\geq 1.$ Utilizing the bounds $t_i\geq 1,$ we get for any fixed $p-$tuple 
$1\leq i_1<i_2<\cdots<i_p,$ the inequality 
$$
 t_{i_1}t_{i_2}\cdots t_{i_p} \geq	 \frac{1}{{p \choose q}} \sum_{\{j_1,j_2,\dots,j_q\}\subset \{i_1,i_2,\dots,i_p\} } t_{j_1}t_{j_2}\cdots t_{j_q},
$$
thus we discover the estimate 
\begin{equation}
\label{4.6}
\frac{1}{{n \choose p}} \sum_{1\leq i_1<i_2<\dots<i_p\leq n} t_{i_1}t_{i_2}\cdots t_{i_p}\geq 
\frac{1}{{p \choose q}{n \choose p}} \sum_{1\leq i_1<i_2<\dots<i_p\leq n} \sum_{\{j_1,j_2,\dots,j_q\}\subset \{i_1,i_2,\dots,i_p\} } t_{j_1}t_{j_2}\cdots t_{j_q}.
\end{equation}
On the other hand any fixed $q-$tuple can be completed to a $p-$tuple in ${n-q \choose p-q}$ different ways, which implies that 
\begin{equation}
\label{4.7} 
\frac{1}{{p \choose q}{n \choose p}} \sum_{1\leq i_1<i_2<\dots<i_p\leq n} \sum_{\{j_1,j_2,\dots,j_q\}\subset \{i_1,i_2,\dots,i_p\} } t_{j_1}t_{j_2}\cdots t_{j_q}=
\frac{{n-q \choose p-q}}{{p \choose q}{n \choose p}} \sum_{1\leq j_1<j_2<\dots<j_q\leq n} t_{j_1}t_{j_2}\cdots t_{j_q}.
\end{equation}
It remains to notice that $\frac{{n-q \choose p-q}}{{p \choose q}{n \choose p}}=\frac{1}{{n \choose q}}.$
\end{proof}

\subsection{Proof of Theorem~\ref{th:3.1}}

\begin{proof} Assume in contradiction that the assertion of Theorem~\ref{th:3.1} fails to hold, thus there exists a form $Q_1\in\cal{C}$ such that $Q_1\neq \alpha Q$ for any $\alpha\in[0,1],$ 
and also satisfying the inequalities 
\begin{equation}
\label{4.8}
0\leq Q_1(\Bx\otimes\By)\leq Q(\Bx\otimes\By),\quad\text{for all}\quad \Bx,\By\in\mathbb R^3.
\end{equation}
It is clear that the form $Q$ is equivalent modulo Null-Lagrangians to the quadratic form 
\begin{equation} 
\label{4.9}
\sum_{i,j=1}^3a_{ij}\xi_{ii}\xi_{jj}+b(\xi_{12}^2+\xi_{21}^2)+c(\xi_{13}^2+\xi_{31}^2)+d(\xi_{23}^2+\xi_{32}^2),
\end{equation}
where we have $a_{ii}=C_{ii}>0$ for $i=1,2,3$ and $b=C_{66}>0,c=C_{55}>0$ and $d=C_{44}>0.$ Therefore we will assume that $Q$ has the form in (\ref{4.9}) in the sequel.  Assume $Q(\Bx\otimes\By)=\Bx T(\By)\Bx^T$ and $Q_1(\Bx\otimes\By)=\Bx T^1(\By)\Bx^T,$ where 

\begin{equation} 
\label{4.10}
Q_1(\bm{\xi})=\sum_{i,j=1}^na_{ij}^1\xi_{ii}\xi_{jj}+b^1(\xi_{12}^2+\xi_{21}^2)+c^1(\xi_{13}^2+\xi_{31}^2)+d^1(\xi_{23}^2+\xi_{32}^2).
\end{equation}
The following polynomial in the variable $\lambda$ is going to be a key factor of the proof:

\begin{align}
\label{4.11}
P(\lambda)&=\det(T(\By)-\lambda T^1(\By))\\ \nonumber
&=\det(T(\By))-\lambda\sum_{i,j=1}^3t^1_{ij}(\By)(\cof(T(\By)))_{ij}+\lambda^2\sum_{i,j=1}^3t_{ij}(\By)(\cof(T^1(\By)))_{ij}-\lambda^3\det(T^1(\By)).
\end{align}
It is clear that due to Lemma~\ref{lem:4.1} we have the inequalities 

\begin{align}
\label{4.12}
0\leq 3\det(T_1(\By))\leq \sum_{i,j=1}^3t_{ij}&(\By)(\cof(T^1(\By)))_{ij}\leq \sum_{i,j=1}^3t^1_{ij}(\By)(\cof(T(\By)))_{ij}\leq 3\det(T(\By)),\\ \nonumber
&\text{for all} \qquad \By\in\mathbb R^3.
\end{align}
Therefore all the polynomials $3\det(T^1(\By))$, $\sum_{i,j=1}^3t_{ij}(\By)(\cof(T^1(\By)))_{ij}$ and  $\sum_{i,j=1}^3t^1_{ij}(\By)(\cof(T(\By)))_{ij},$ being between zero and the extremal polynomial $3\det(T(\By)),$ must be scalar multiples of $\det(T(\By)),$ i.e., we have 
\begin{align}
\label{4.13}
\det(T^1(\By))&=\alpha \det(T(\By)),\\ \nonumber
\sum_{i,j=1}^3t_{ij}(\By)(\cof(T^1(\By)))_{ij}&=\beta\det(T(\By)),\\ \nonumber
\sum_{i,j=1}^3t^1_{ij}(\By)(\cof(T(\By)))_{ij}&=\gamma \det(T(\By)),\\ \nonumber
\text{for some}\quad \alpha, \beta, \gamma &\geq 0.
\end{align}
Consequently we get from (\ref{4.11}) and (\ref{4.13}) the identity

\begin{equation}
\label{4.14}
\det(T(\By)-\lambda T^1(\By))=(1-\gamma\lambda+\beta\lambda^2-\alpha\lambda^3)\det(T(\By)),
\quad\text{for all}\quad \By\in \mathbb R^3, \ \lambda\in\mathbb R.
\end{equation}
Owing to the form (\ref{4.9}) of the quadratic for $Q(\bm{\xi}),$ we have that

\begin{equation}
\label{4.15} 
T(\By)=
\begin{bmatrix}
a_{11}y_1^2+by_2^2+cy_3^2 & a_{12}y_1y_2 & a_{13}y_1y_3\\[2ex]
a_{12}y_1y_2 & a_{22}y_2^2+by_1^2+dy_3^2 & a_{23}y_2y_3\\[2ex]
a_{13}y_1y_3 & a_{23}y_2y_3 & a_{33}y_3^2+cy_1^2+dy_2^2
\end{bmatrix}.
\end{equation}
Next we have by direct calculation (or using maple), that
\begin{align}
\label{4.16}
&\det(T(\By))\\ \nonumber
&=(a_{11}bc)y_1^6+(a_{22}bd)y_2^6+(a_{33}cd)y_3^6\\ \nonumber
&+(a_{11}bd+a_{11}a_{22}c+b^2c-a_{12}^2c)y_1^4y_2^2\\ \nonumber
&+(a_{22}bc+a_{11}a_{22}d+b^2d-a_{12}^2d)y_2^4y_1^2\\ \nonumber
&+(a_{11}cd+a_{11}a_{33}b+c^2b-a_{13}^2b)y_1^4y_3^2\\ \nonumber
&+(a_{33}bc+a_{11}a_{33}d+c^2d-a_{13}^2d)y_3^4y_1^2\\ \nonumber
&+(a_{22}cd+a_{22}a_{33}b+d^2b-a_{23}^2b)y_2^4y_3^2\\ \nonumber
&+(a_{33}bd+a_{22}a_{33}c+d^2c-a_{23}^2c)y_3^4y_2^2\\ \nonumber
&+(a_{11}a_{22}a_{33}+2a_{12}a_{13}a_{23}-a_{11}a_{23}^2-a_{22}a_{13}^2-a_{33}a_{12}^2+a_{11}d^2+a_{22}c^2+a_{33}b^2+2bcd)y_1^2y_2^2y_3^2.
\end{align}
Due to the strict inequalities (\ref{3.2}), we can denote $\frac{a_{ii}^1}{a_{ii}}=q_{ii},$ $i=1,2,3,$ and $\frac{b^1}{b}=q_b,$ $\frac{c^1}{c}=q_c$, $\frac{d^1}{d}=q_d.$ The identity (\ref{4.14}) tells us that the quotient of the coefficients of any monomials $y_1^{\alpha_1}y_2^{\alpha_2}y_3^{\alpha_3}$ and 
$y_1^{\beta_1}y_2^{\beta_2}y_3^{\beta_3}$ in $\det(T(\By)-\lambda T^1(\By))$ is exactly that of in $\det(T(\By)),$ thus considering the coefficients of $y_1^6,$ $y_2^6$ and $y_3^6$ we get the following set of identities:

\begin{align}
\label{4.17}
(1-q_{11}\lambda)(1-q_c\lambda)&=(1-q_{22}\lambda)(1-q_d\lambda),\\ \nonumber
(1-q_{11}\lambda)(1-q_b\lambda)&=(1-q_{33}\lambda)(1-q_d\lambda),\\ \nonumber
(1-q_{22}\lambda)(1-q_b\lambda)&=(1-q_{33}\lambda)(1-q_c\lambda),\\ \nonumber
\text{for all}\quad \lambda &\in\mathbb R.
\end{align}
The conditions in (\ref{4.17}) imply that the roots of the polynomials on the right and the left are the same, thus we get the set equalities
\begin{equation}
\label{4.18}
\{q_{11},q_c\}=\{q_{22},q_d\},\quad \{q_{11},q_b\}=\{q_{33},q_d\},\quad \{q_{22},q_b\}=\{q_{33},q_c\}.
\end{equation}
The are three main cases possible, the remaining ones being similar.\\
\textbf{Case1.} $q_{11}=q_{22}=q_{33}=s,$ and $q_b=q_c=q_d=t.$ \\
\textbf{Case2.} $q_{11}=q_{22}=q_{c}=q_{d}=s,$ and $q_{33}=q_b=t.$ \\
\textbf{Case3.} $q_{11}=q_d=s,$ $q_{22}=q_c=t,$ and $q_{33}=q_b=u.$ \\
The goal is to prove that $s=t$ in the first two cases and $s=t=u$ in the third case. We start with Case2 as Case1 is the most involved consisting of several sub-cases.\\
\textbf{Case2.}  as In this case the coefficient of $y_1^6$ in $\det(T(\By)-\lambda T^1(\By))$ is a scalar multiple of $(1- s\lambda)^2(1-t\lambda),$ thus also the coefficient of $y_2^4y_3^2$ in $\det(T(\By)-\lambda T^1(\By))$ has to have the same property. The first three summands in the coefficient of $y_2^4y_3^2$ in 
$\det(T(\By)-\lambda T^1(\By))$ have the factor $1-s\lambda,$ thus if we assume that $s\neq t,$ then the polynomial $(a_{23}-a_{23}^1\lambda)^2$ must have the factor $1-s\lambda,$ which implies that $a_{23}\neq 0$ and $q_{23}=\frac{a_{23}^1}{a_{23}}=s.$ Consequently all the three summands in the coefficient of $y_2^4y_3^2$ in $\det(T(\By)-\lambda T^1(\By))$
are multiples of $(1-s\lambda)^2$ except the second one, which is a scalar multiple of $(1-s\lambda)(1-t\lambda)^2,$ which again yields $s=t.$\\
\textbf{Case3.} In this case the coefficient of $y_1^6$ in $\det(T(\By)-\lambda T^1(\By))$ is a scalar multiple of $(1- s\lambda)(1-t\lambda)(1-u\lambda).$ Considering the coefficient of $y_3^4y_1^2$ in $\det(T(\By)-\lambda T^1(\By)),$ we have that the summands in it except the third one, which is $c^2d(1-t\lambda)^2(1-u\lambda),$ are divisible by
$(1-s\lambda),$ thus we obtain that either $t=s$ or $u=s.$ It is clear that both cases reduce to Case2, thus we finally end up with $s=t=u.$ \\
\textbf{Case1.} We aim to prove that again $t=s.$ Assume in contradiction that $t\neq s,$ which we will use in the subsequent calculations. Unlike the other two cases this one is somewhat more involved, consisting of several possible sub-cases. We first consider the coefficients of $y_1^6$ and $y_1^4y_2^2$ in both polynomials $\det(T(\By)-\lambda T^1(\By))$ and $\det(T(\By).$ 
Taking into account the condition (\ref{4.14}), we get the polynomial identity 
\begin{align*}
a_{11}bd(1-s\lambda)(1-&t\lambda)^2+a_{11}a_{22}c(1-s\lambda)^2(1-t\lambda)+b^2c(1-t\lambda)^3-c(1-t\lambda)(a_{12}-a_{12}^1\lambda)^2\\
&=(a_{11}bd+a_{11}a_{22}c+b^2c-a_{12}^2c)(1-s\lambda)(1-t\lambda)^2,\quad\text{for all}\quad\lambda\in (0,1),
\end{align*}
which gives the condition 
\begin{align*}
a_{11}a_{22}(t-s)(1-s\lambda)-b^2(t-s)(1-t\lambda)-a_{12}^2(s+t)-2a_{12}a_{12}^1+((a_{12}^1)^2-sta_{12}^2)\lambda=0,\\
\quad\text{for all}\quad\lambda\in (0,1).
\end{align*}
The last polynomial equality yields the two conditions:
\begin{align}
\label{4.19}
-a_{11}a_{22}s(t-s)+b^2t(t-s)-((a_{12}^1)^2-a_{12}^2st)=0,\\ \nonumber
a_{11}a_{22}(t-s)-b^2(t-s)-a_{12}^2(s+t)+2a_{12}a_{12}^1)=0.
\end{align}
Next we first multiply the second equality in (\ref{4.19}) by $s$ and add it to the first one, and then do the same with $t$ instead to $s$ to get the following two new equalities:
\begin{align}
\label{4.20}
b^2(t-s)^2=(a_{12}^1-sa_{12})^2,\\ \nonumber
a_{11}a_{22}(t-s)^2=(a_{12}^1-ta_{12})^2.
\end{align}
Similarly, by considering the coefficient of $y_1^6$ together with the coefficient of $y_1^4y_3^2$ and $y_2^4y_3^2$ respectively, we obtain four more similar identities:
\begin{align*}
c^2(t-s)^2=(a_{13}^1-sa_{13})^2,\\ 
a_{11}a_{33}(t-s)^2=(a_{13}^1-ta_{13})^2,\\
d^2(t-s)^2=(a_{23}^1-sa_{23})^2,\\ 
a_{22}a_{33}(t-s)^2=(a_{23}^1-ta_{23})^2,
\end{align*} 
thus we have the following set of six identities: 
\begin{align}
\label{4.21}
b^2(t-s)^2=(a_{12}^1-sa_{12})^2,\\ \nonumber
a_{11}a_{22}(t-s)^2=(a_{12}^1-ta_{12})^2,\\ \nonumber
c^2(t-s)^2=(a_{13}^1-sa_{13})^2,\\  \nonumber
a_{11}a_{33}(t-s)^2=(a_{13}^1-ta_{13})^2,\\ \nonumber
d^2(t-s)^2=(a_{23}^1-sa_{23})^2,\\  \nonumber
 a_{22}a_{33}(t-s)^2=(a_{23}^1-ta_{23})^2.
\end{align}
Next we treat the off diagonal entries $a_{ij}$ and $a_{ij}^1,$ ($i\neq j$) as unknowns in (\ref{4.21}) and solve each pair of identities for them, having 
three systems of two equations with two unknowns in each. We have three possibilities for each system of two equations, namely for instance for the first two equations in (\ref{4.21}) we have 
$$a_{11}a_{22}(a_{12}^1-sa_{12})^2=b^2(a_{12}^1-ta_{12})^2,$$ 
which splits into the two cases 
\begin{equation}
\label{4.22}
\sqrt{a_{11}a_{22}}(a_{12}^1-sa_{12})=\pm b(a_{12}^1-ta_{12}).
\end{equation}
The case with the minus sign on the right gives 
$$a_{12}^1=\frac{a_{12}(tb+s\sqrt{a_{11}a_{22}})}{b+\sqrt{a_{11}a_{22}}},$$ 
hence plugging this formula for $a_{12}^1$ back into one of the equations and taking into account the fact that $s-t\neq 0,$ we get the condition 
$$|a_{12}|=b+\sqrt{a_{11}a_{22}}.$$
This being said, the case when we have a negative sign on the right side of the equality in (\ref{4.22}) yields the solution
\begin{equation}
\label{4.23}
|a_{12}|=b+\sqrt{a_{11}a_{22}},\qquad a_{12}^1=\sigma_{12}(tb+s\sqrt{a_{11}a_{22}}),
\end{equation} 
where $\sigma_{ij}$ denotes the sign of $a_{ij}$ for $i\neq j.$ In the case when we have a positive sign on the right-hand side of the equality in (\ref{4.22}), we get 
$$a_{12}^1(b-\sqrt{a_{11}a_{22}})=a_{12}(tb-s\sqrt{a_{11}a_{22}}),$$
thus in the case $b=\sqrt{a_{11}a_{22}}$ we obtain $a_{12}=0.$ For $a_{12}^1$ we find from (\ref{4.19}) that $|a_{12}^1|=|b(t-s)|.$
Now in the case $b\neq \sqrt{a_{11}a_{22}}$ we clearly obtain the solution
$$
|a_{12}|=|b-\sqrt{a_{11}a_{22}}|,\qquad a_{12}^1=\frac{a_{12}(tb-s\sqrt{a_{11}a_{22}})}{b-\sqrt{a_{11}a_{22}}},
$$
thus combining the two possibilities we get 

\begin{equation}
\label{4.24}
\begin{cases}
|a_{12}|=|b-\sqrt{a_{11}a_{22}}|,\qquad a_{12}^1=\frac{a_{12}(tb-s\sqrt{a_{11}a_{22}})}{b-\sqrt{a_{11}a_{22}}} & if \quad b\neq \sqrt{a_{11}a_{22}},\\
a_{12}=0,\qquad |a_{12}^1|=|b(t-s)| & if \quad b=\sqrt{a_{11}a_{22}}.
\end{cases}
\end{equation} 
Next we note that in the case when we have the formula (\ref{4.23}), we get 
\begin{equation}
\label{4.25}
a_{12}-a_{12}^1\lambda=\pm (b(1-t\lambda)+\sqrt{a_{11}a_{22}}(1-s\lambda)),
\end{equation}
and in the case when we have (\ref{4.24}), we get 
\begin{equation}
\label{4.26}
a_{12}-a_{12}^1\lambda=\pm (b(1-t\lambda)-\sqrt{a_{11}a_{22}}(1-s\lambda)).
\end{equation}
The last step is the consideration of the coefficient of $y_1^2y_2^2y_3^2,$ which would then lead to a contradiction. First recall that $\det(T(\By))$ can not be a perfect square as understood in 
Remark~\ref{rem:3.2}. Assume now $s=0.$ In this case the bi-quadratic form $Q_1(\Bx\otimes\By )$ will be separated into parts depending only on pairs $(\xi_{ij},\xi_{ji}),$ $1\leq i<j\leq 3,$ thus 
the condition $Q_1(\Bx\otimes\By )\geq 0$ for all $\Bx,\By\in\mathbb R^3$ implies that $Q_1$ is convex. However, since $\det(T(\By))$ is not a perfect square, Theorem~5.1
in [\ref{bib:Har.Mil.2}] (see also [\ref{bib:Har.Mil.3}, Theorem~2.5]) implies that (\ref{4.9}) can not be fulfilled unless $Q_1=\alpha Q$ for some $\alpha\in[0,1],$ 
which would immediately finish the proof of Theorem~\ref{th:3.1}. Assume now $t=0.$ In this case too the inequality $Q_1(\Bx\otimes\By )\geq 0$ for all $\Bx,\By\in\mathbb R^3$ 
is equivalent to the fact that $Q_1$ is convex (as $Q_1$ depends only on the variables $\xi_{11},\xi_{22},\xi_{33}$), thus we are again done. 
We now have that $s,t>0$ and $s\neq t,$ thus the linear functions $X=1-s\lambda$ and $Y=1-t\lambda$ are non-constant and linearly independent.   
We now consider the coefficient (denoted by A(X,Y)) of $y_1^2y_2^2y_3^2$ in $\det(T(\By)-\lambda T^1(\By)),$ which according to (\ref{4.25}) and (\ref{4.26}) has the form 

\begin{align}
\label{4.28}
A(X,Y)&=a_{11}a_{22}a_{33}X^3 \pm 2(bY\pm\sqrt{a_{11}a_{22}}X)(cY\pm\sqrt{a_{11}a_{33}}X)(dY\pm\sqrt{a_{22}a_{33}}X)\\ \nonumber
&-a_{11}X (dY\pm\sqrt{a_{22}a_{33}}X)^2-a_{22}X (cY\pm\sqrt{a_{11}a_{33}}X)^2-a_{33}X (bY\pm\sqrt{a_{11}a_{22}}X)^2\\ \nonumber
&+(a_{11}d^2+a_{22}c^2+a_{33}b^2)XY^2+2bcdY^3.
\end{align}
On the other hand like the coefficient of $y_1^6,$ the polynomial $A(X,Y)$ must have the form $BXY^2.$ As $X$ and $Y$ are linearly independent non-constant linear functions of $\lambda,$ 
and the equality $A(X,Y)=BXY^2$ must be fulfilled for all $\lambda\in[0,1],$ we have to get $BXY^2$ algebraically after doing all the algebraic operations in (\ref{4.28}). 
This means that first of all we have to have a negative sign in front of $2(bY\pm\sqrt{a_{11}a_{22}}X)(cY\pm\sqrt{a_{11}a_{33}}X)(dY\pm\sqrt{a_{22}a_{33}}X)$ so that the coefficient of 
$Y^3$ vanishes. Next, by looking at the coefficient of $X^3$ in $A(X,Y),$ we get that there must be exactly one multiplier with "minus" and two with "plus" in the product 
$(bY\pm\sqrt{a_{11}a_{22}}X)(cY\pm\sqrt{a_{11}a_{33}}X)(dY\pm\sqrt{a_{22}a_{33}}X).$ Therefore, by symmetry, we can without loss of generality assume that we have 
\begin{align}
\label{4.29}
|a_{12}-a_{12}^1\lambda|&=|b(1-t\lambda)-\sqrt{a_{11}a_{22}}(1-s\lambda)|,\\ \nonumber
|a_{13}-a_{13}^1\lambda|&=|c(1-t\lambda)+\sqrt{a_{11}a_{33}}(1-s\lambda)|,\\ \nonumber
|a_{23}-a_{23}^1\lambda|&=|d(1-t\lambda)+\sqrt{a_{22}a_{33}}(1-s\lambda)|.
\end{align}
Taking into account the algebraic negative sign in front of $2(bY\pm\sqrt{a_{11}a_{22}}X)(cY\pm\sqrt{a_{11}a_{33}}X)(dY\pm\sqrt{a_{22}a_{33}}X)$ and (\ref{4.29}), 
we are left with the following three main possibilities (again due to symmetries):\\
\textbf{Case1a.} $a_{12}^1=tb-s\sqrt{a_{11}a_{22}},\quad a_{13}^1=-(tc+s\sqrt{a_{11}a_{33}}),\quad a_{23}^1=td+s\sqrt{a_{22}a_{33}}.$\\
\textbf{Case1b.} $a_{12}^1=-(tb-s\sqrt{a_{11}a_{22}}),\quad a_{13}^1=tc+s\sqrt{a_{11}a_{33}},\quad a_{23}^1=td+s\sqrt{a_{22}a_{33}}.$\\
\textbf{Case1c.} $a_{12}^1=-(tb-s\sqrt{a_{11}a_{22}}),\quad a_{13}^1=-(tc+s\sqrt{a_{11}a_{33}}),\quad a_{23}^1=-(td+s\sqrt{a_{22}a_{33}}).$\\
Note that in all three cases, the number of negative $\sqrt{a_{ii}a_{jj}}$,  $i\neq j$ is even, thus remembering the definition of $s$ and $t$ we obtain that $Q_1$ is a sum of squares, 
like for instance in Case1a we have 
$$Q_1(\xi)=s(-\sqrt{a_{11}}\xi_{11}+\sqrt{a_{22}}\xi_{22}+\sqrt{a_{33}}\xi_{3})^2+tb(\xi_{12}+\xi_{21})^2+tc(\xi_{12}-\xi_{21})^2+td(\xi_{12}+\xi_{21})^2.$$
However, this leads to a contradiction as again owing to [\ref{bib:Har.Mil.2}, Theorem~5.1], we have that the form $Q(\xi)$ is an extremal in the sense of Milton (recall that $\det{T(\By)}$ is not a perfect square), i.e., no convex form can be subtracted from it while preserving the quasiconvexity of it. This completes the proof of Theorem~\ref{th:3.1}.

\end{proof}

\begin{remark}
The tools of the proof of Theorem~\ref{th:3.1} are quite robust and it can be easily checked that quadratic forms having for instance the form 
\begin{equation}
\label{4.30}
Q(\bm{\xi})=\sum_{i,j=1}^3a_{ij}\xi_{ii}\xi_{jj}+b\xi_{12}^2+c\xi_{23}^2+d\xi_{31}^2 
\end{equation}
can also be treated by the same technique. It has been proven in [\ref{bib:Har.Mil.1}], that there is an extreme point of form (\ref{4.19}), namely the quadratic form 
\begin{equation}
\label{4.31}
F(\bm{\xi})=\xi_{11}^2+\xi_{22}^2+\xi_{33}^2+\xi_{12}^2+\xi_{23}^2+\xi_{31}^2-2(\xi_{11}\xi_{22}+\xi_{22}\xi_{33}+\xi_{33}\xi_{11}). 
\end{equation}
which is not only an extreme point of the cone of quadratic forms that have the same form (\ref{4.19}), but also it is an extreme point of the cone of all quasiconvex quadratic forms. Example (\ref{4.20}) is the first one as such in the literature. Also, quadratics forms having the more general form 
\begin{equation}
\label{4.32}
Q(\bm{\xi})=\sum_{i,j=1}^3a_{ij}\xi_{ii}\xi_{jj}+b_1\xi_{12}^2+c_1\xi_{23}^2+d_1\xi_{31}^2+b_2\xi_{21}^2+c_2\xi_{32}^2+d_2\xi_{13}^2 
\end{equation}
that involves orthotropic ones, seem to fit into the above analysis, but due to the bigger number of parameters involved, and thus possible cases to consider, we prefer not to present a detailed analysis here. 
\end{remark}

\section*{Acknowledgements}
We would like to thank one of the referees for spotting typos and pointing some relevant papers in the language of positive biquadratic forms,  
addressing the comments has substantially improved the presentation of the manuscript. We also thank the Isaac Newton Institute for Mathematical 
Sciences, Cambridge, UK for great hospitality, where part of the research in the manuscript has been carried out while the author was 
a visitor of a scientific program there.

\end{document}

%% file: mymacros.tex
\usepackage{amssymb,bm}
\usepackage{amsmath}
\usepackage{amsthm}
\usepackage{graphicx}
\usepackage[active]{srcltx} % Includes line number info into dvi file
\usepackage{hyperref}
\hypersetup{pdfborder=0 0 0}

\newtheorem{theorem}{{\sc Theorem}}[section]

\newtheorem{lemma}[theorem]{{\sc Lemma}}

\newtheorem{remark}[theorem]{Remark}

\newcommand{\cof}{\mathrm{cof}}

\def\XXint#1#2#3{{\setbox0=\hbox{$#1{#2#3}{\int}$ }
\vcenter{\hbox{$#2#3$ }}\kern-.6\wd0}}

\bmdefine\BGa{\alpha}
\bmdefine\BGb{\beta}
\bmdefine\BGd{\delta}
\bmdefine\BGe{\epsilon}
\bmdefine\BGve{\varepsilon}
\bmdefine\BGf{\phi}
\bmdefine\BGvf{\varphi}
\bmdefine\BGg{\gamma}
\bmdefine\BGc{\chi}
\bmdefine\BGi{\iota}
\bmdefine\BGk{\kappa}
\bmdefine\BGl{\lambda}
\bmdefine\BGn{\eta}
\bmdefine\BGm{\mu}
\bmdefine\BGv{\nu}
\bmdefine\BGp{\pi}
\bmdefine\BGth{\theta}
\bmdefine\BGvth{\vartheta}
\bmdefine\BGr{\rho}
\bmdefine\BGvr{\varrho}
\bmdefine\BGs{\sigma}
\bmdefine\BGvs{\varsigma}
\bmdefine\BGt{\tau}
\bmdefine\BGj{\tau}
\bmdefine\BGu{\upsilon}
\bmdefine\BGo{\omega}
\bmdefine\BGx{\xi}
\bmdefine\BGy{\psi}
\bmdefine\BGz{\zeta}
\bmdefine\BGD{\Delta}
\bmdefine\BGF{\Phi}
\bmdefine\BGG{\Gamma}
\bmdefine\BGL{\Lambda}
\bmdefine\BGP{\Pi}
\bmdefine\BGT{\Theta}
\bmdefine\BGS{\Sigma}
\bmdefine\BGU{\Upsilon}
\bmdefine\BGO{\Omega}
\bmdefine\BGX{\Xi}
\bmdefine\BGY{\Psi}

\bmdefine\BCA{{\mathcal A}}
\bmdefine\BCB{{\mathcal B}}
\bmdefine\BCC{{\mathcal C}}
\bmdefine\BCD{{\mathcal D}}
\bmdefine\BCE{{\mathcal E}}
\bmdefine\BCF{{\mathcal F}}
\bmdefine\BCG{{\mathcal G}}
\bmdefine\BCH{{\mathcal H}}
\bmdefine\BCI{{\mathcal I}}
\bmdefine\BCJ{{\mathcal J}}
\bmdefine\BCK{{\mathcal K}}
\bmdefine\BCL{{\mathcal L}}
\bmdefine\BCM{{\mathcal M}}
\bmdefine\BCN{{\mathcal N}}
\bmdefine\BCO{{\mathcal O}}
\bmdefine\BCP{{\mathcal P}}
\bmdefine\BCQ{{\mathcal Q}}
\bmdefine\BCR{{\mathcal R}}
\bmdefine\BCS{{\mathcal S}}
\bmdefine\BCT{{\mathcal T}}
\bmdefine\BCU{{\mathcal U}}
\bmdefine\BCV{{\mathcal V}}
\bmdefine\BCW{{\mathcal W}}
\bmdefine\BCX{{\mathcal X}}
\bmdefine\BCY{{\mathcal Y}}
\bmdefine\BCZ{{\mathcal Z}}

\bmdefine\Bzr{ 0}
\bmdefine\Ba{ a}
\bmdefine\Bb{ b}
\bmdefine\Bc{ c}
\bmdefine\Bd{ d}
\bmdefine\Be{ e}
\bmdefine\Bf{ f}
\bmdefine\Bg{ g}
\bmdefine\Bh{ h}
\bmdefine\Bi{ i}
\bmdefine\Bj{ j}
\bmdefine\Bk{ k}
\bmdefine\Bl{ l}
\bmdefine\Bm{ m}
\bmdefine\Bn{ n}
\bmdefine\Bo{ o}
\bmdefine\Bp{ p}
\bmdefine\Bq{ q}
\bmdefine\Br{ r}
\bmdefine\Bs{ s}
\bmdefine\Bt{ t}
\bmdefine\Bu{ u}
\bmdefine\Bv{ v}
\bmdefine\Bw{ w}
\bmdefine\Bx{ x}
\bmdefine\By{ y}
\bmdefine\Bz{ z}
\bmdefine\BA{ A}
\bmdefine\BB{ B}
\bmdefine\BC{ C}
\bmdefine\BD{ D}
\bmdefine\BE{ E}
\bmdefine\BF{ F}
\bmdefine\BG{ G}
\bmdefine\BH{ H}
\bmdefine\BI{ I}
\bmdefine\BJ{ J}
\bmdefine\BK{ K}
\bmdefine\BL{ L}
\bmdefine\BM{ M}
\bmdefine\BN{ N}
\bmdefine\BO{ O}
\bmdefine\BP{ P}
\bmdefine\BQ{ Q}
\bmdefine\BR{ R}
\bmdefine\BS{ S}
\bmdefine\BT{ T}
\bmdefine\BU{ U}
\bmdefine\BV{ V}
\bmdefine\BW{ W}
\bmdefine\BX{ X}
\bmdefine\BY{ Y}
\bmdefine\BZ{ Z}